\documentclass[11pt,reqno]{amsart}
\usepackage{amsmath}
\usepackage{amsfonts}
\usepackage{amssymb}
\usepackage{bm}
\usepackage[active]{srcltx}
\usepackage{cases}
\usepackage{stmaryrd}
\usepackage{xcolor}
\usepackage{graphics}
\usepackage{graphicx}
\hoffset=- 2cm \voffset=- 0cm
\theoremstyle{plain}
\newtheorem{Thm}{Theorem}

\newtheorem{Lem}[Thm]{Lemma}
\newtheorem{Cor}[Thm]{Corollary}

\newcommand {\p}{\partial}
\newcommand{\q}{\quad}
\newcommand{\qq}{\qquad}

\newcommand{\eq}{\begin{equation}}
\newcommand{\eeq}{\end{equation}}

\def\a{\alpha}

\def\k{\kappa}

\def\p{\partial}

\def\O{\Omega}

\def\B{\bold B}

\def\E{\bold E}

\def\f{\bold f}
\def\F{\bold F}

\def\H{\bold H}

\def\h{\bold h}
\def\n{\bold n}
\def\r{\bold r}

\def\u{\bold u}

\def\w{\bold w}

\def\z{\bold z}

\def\0{\bold 0}

\def\mB{\mathcal B}

\def\mF{\mathcal F}

\def\Q1{\overline{Q1}}

\def\mU{\mathcal U}

\def\q{\quad}
\def\qq{\qquad}
\def\qqq{\qq\qq}

\def\ol{\overline}
\def\v{\vskip}

\def\curl{\text{\rm curl\,}}
\def\div{\text{\rm div\,}}
\def\spt{\text{\rm spt}}

\errorcontextlines=0 \numberwithin{equation}{section}
\numberwithin{Thm}{section}

\begin{document}
\large
\title[Parabolic Curl System]
{Regularity of a Parabolic System Involving Curl}

\author[X.B.Pan]{Xing-Bin Pan }

\address{School of Science and Engineering,
The Chinese University of Hong Kong (Shenzhen), Shenzhen 518172, Guangdong, China; and School of Mathematical Sciences, East China Normal University, Shanghai 200062, China}

\email{panxingbin@cuhk.edu.cn; xbpan@math.ecnu.edu.cn}

\thanks{ }



\begin{abstract} This note presents a regularity result with proof for an initial-boundary value problem of a  linear parabolic system involving curl of the unknown vector field, subjected to the boundary condition of prescribing the tangential component of the solution.
\end{abstract}

\maketitle

\tableofcontents

\section{Introduction}

We are interested in the regularity theory of linear parabolic systems involving curl. We believe that the regularity results are well-known to the experts. However it is difficult to find the statements with complete proofs in the literature. Therefore we wish to write out the conclusions with proofs, for our later references.
We wish to start our program with the equation of the following form
\eq\label{eqB}
\left\{\aligned
&{\p\u\over\p t}+a\,\curl^2\u+\mB\,\curl\u+c\u=\f,\q\div\u=0,\q\;& (t,x)\in Q_T,\\
&\u_T=\0,\qqq\, &(t,x)\in S_T,\\
&\u(0,x)=\u^0,\qq\;\;\;& x\in \O,
\endaligned\right.
\eeq
where $a, c$ are scalar functions, $\mB$ is a matrix-valued function, $Q_T=(0,T]\times\O$ with $\O$ being a bounded domain in $\Bbb R^3$ with smooth boundary $\p\O$, $S_T=(0,T]\times\p\O$. We denote $\curl\u\equiv \nabla\times\u$, and denote by $\u_T$ the tangential component of $\u$ at boundary $\p\O$, namely $\u_T=(\nu\times\u)\times\nu$, where $\nu$ is the unit outer normal vector of $\p\O$.
In this paper, we use $M(3)$ to denote the set of all $3\times 3$ matrices, and let
$$
C^{k+\a}_{t0}(\ol{\O},\div0)=\{\w\in C^{k+\a}(\ol{\O},\Bbb R^3):~ \div\w=0\;\;\text{in }\O,\;\; \w_T=\0\;\;\text{on }\p\O\}.
$$

Note that the boundary condition in \eqref{eqB} is to prescribe the tangential component of the solution, and it makes \eqref{eqB} significantly different to the usual parabolic equation with Dirichlet boundary condition which prescribes the full trace.
The regularity of weak solutions of \eqref{eqB} will be used in \cite{KP} to establish existence and regularity of weak solutions of the time-dependent model of Meissner states of superconductors.

\begin{Thm}\label{Thm1} Let $\O$ be a bounded domain in $\Bbb R^3$ with a $C^{3+\a}$ boundary, $Q_{T}=(0,T]\times \O$. Assume that
$$\aligned
& 0<\a<1,\qq a,\; c\in C^{\a,\a/2}(\overline{Q}_T),\qq a(t,x)\geq a_0>0,\\
& \mB\in C^{\a,\a/2}(\overline{Q}_T,M(3)),\qq
\f\in C^{\a,\a/2}(\overline{Q}_T,\Bbb R^3),\qq
 \u^0\in C^{2+\a}_{t0}(\overline{\O},\div0),
\endaligned
$$
and
\eq\label{adm}
a(0,x)[\curl^2\u^0(x)]_T+[\mB(0,x)\curl\,\u^0(x)]_T=[\f(0,x)]_T,\q x\in\p\O.
\eeq
If $\u$ is a weak solution of \eqref{eqB} on $Q_{T}$, then $\u\in C^{2+\a,1+\a/2}(\overline{Q}_T)$ and
$$
\|\u\|_{C^{2+\a,1+\a/2}(\overline{Q}_T)}\leq C\{\|\f\|_{C^{\a,\a/2}(Q_{T})}+\|\u^0\|_{C^{2+\a}(\overline{\O})}\},
$$
where $C$ depends only on $\O, T, \a$ and the $C^{\a,\a/2}(\overline{Q}_T)$ norm of $a, b, B$.
\end{Thm}

In \eqref{adm} we use $[\cdot]_T$ to denote the tangential component of the enclosed vector. Let us mention that the assumption $\u^0\in C^{2+\a}_{t0}(\ol{\O},\div0)$ implies that $\u^0_T=\0$ for $x\in \p\O$, which is consistent with the boundary condition $\u_T=\0$. This together with the assumption \eqref{adm} consists of the compatibility condition for the problem \eqref{eqB}.

\section{Estimates Near Flat Boundary}

\subsection{$W^{2,1,q}$-estimates}\

We consider regularity of weak solutions of \eqref{eqB}, where $a, c$ are scalar functions and $\mB$ is a matrix-valued function.
Let $\u$ be a weak solution of \eqref{eqB}. Then $\u\in L^2(0,T; H^1(\O,\Bbb R^3))$. To get higher regularity of the solutions, one may first use the difference method to show that $\u\in L^2(0,T; H^2(\O,\Bbb R^3))$ and $\p_t\u\in L^2(0,T; L^2(\O,\Bbb R^3))$, then show $\u$ is of $C^{2+\a, 1+\a/2}$. Here we use the different approach. We shall start with a weak solution $\u\in L^2(0,T; H^1(\O,\Bbb R^3))$ and show directly $\u$ is of $C^{2+\a,1+\a/2}$.

We shall derive the a priori estimates for smooth functions. Then the regularity of weak solutions follow from the estimates.

By considering cut-off, we only need to examine regularity near boundary. We start with a flat boundary. Denote by $B^+_R$ the upper half ball with center at the origin and radius $R$, and
$$
\Sigma_R=\{x=(x_1,x_2,0):~ |x|<R\}.
$$
Let
$$
Q_{R,T}=(0,T]\times B^+_R,\qq \Gamma_{R,T}=(0,T]\times \Sigma_R.
$$
With the divergence-free condition, \eqref{eqB} can be written in the following form
\eq\label{eqB.1}
\left\{\aligned
&{\p\u\over\p t}-a\,\Delta\u+\mB\,\curl\u+c\u=\f,\q \div\u=0,\q\;& (t,x)\in Q_{R,T},\\
&\u_T=\0,\qqq\, &(t,x)\in \Gamma_{R,T},\\
&\u(0,x)=\u^0,\qq\;\;\;& x\in B^+_R.
\endaligned\right.
\eeq
As mentioned in the introduction, the boundary condition in \eqref{eqB.1} is to prescribe the tangential component, but not the full trace, of the solution. As such, the regularity of \eqref{eqB.1} is not a direct consequence of the regularity theory of the classical initial-Dirichlet boundary problem of parabolic equations.

The compatibility condition \eqref{adm} can be written as
\eq\label{adm2}
-a(0,x)[\Delta  \u^0(x)]_T+[\mB(0,x)\curl\,\u^0(x)]_T=[\f(0,x)]_T,\q x\in\p\O.
\eeq
Note that $\Gamma_{R,T}$ is the flat part of the parabolic boundary of $Q_{R,T}$. We shall establish the following local estimate:

\begin{Lem}\label{Lem2.1}
Assume that
$$\aligned
&a,\; c\in C^{\a,\a/2}(\overline{Q}_{R_0,T}),\qq a(t,x)\geq a_0>0,\qq 1<q<\infty,\\
&\mB\in C^{\a,\a/2}(\overline{Q}_{R_0,T},M(3)),\qq \f\in L^q(Q_{R_0,T},\Bbb R^3),\\
& \u^0\in W^{2,q}(B_{R_0}^+,\Bbb R^3),\qq \div\u^0=0\;\;\text{\rm in }B_{R_0},\qq \u^0_T=\0\;\;\text{\rm on }\Sigma_{R_0},
\endaligned
$$
and assume \eqref{adm} holds.
If $\u$ is a weak solution of \eqref{eqB.1} on $Q_{R_0,T}$, then for any $0<R<R_0$ we have $\u\in W^{2,q}(Q_{R,T},\Bbb R^3)$ and
$$
\|\u\|_{W^{2,1,q}(Q_{R,T})}\leq C\{\|\f\|_{L^q(Q_{R_0,T})}+\|\nabla_x\u\|_{L^2(Q_{R_0,T})}+\|\u\|_{L^2(Q_{R_0,T})}+\|\u^0\|_{W^{2,q}(B_{R_0}^+)}\},
$$
where $C$ depends only on $R_0$, $R$, $T$, $\mB$, $a$, $c$, $q$.
\end{Lem}

\begin{proof}
Let $\u$ be a solution of \eqref{eqB.1}.
Write
$$
\u=(u_1,u_2,u_3)^t,\qq \f=(f_1,f_2,f_3)^t,\qq
\mB\,\curl\u=\h=(h_1,h_2,h_3)^t.
$$
$u_1, u_2$ correspond the tangential component of $\u$ and $u_3$ corresponds to the normal component.
Recall the formula (see \cite[p.210]{DaL})
$$
\div\u=\div_\Gamma(\pi\u)+2(\nu\cdot\u)H(x)+{\p\over\p\nu}(\nu\cdot\u).
$$
In the above $\pi\u$ denotes the tangential component of $\u$ on the domain boundary, $\div_\Gamma$ denotes the surface divergence, and $H(x)$ is the mean curvature of the domain boundary. Applying the above equality on the flat part of the boundary $\Sigma_R$ where $H(x)\equiv 0$ we see that, the boundary condition $\u_T=\0$ together with the divergence-free condition $\div\u=0$ implies the Neumann boundary condition for $u_3$.
In fact, for $x\in \Sigma_R$ we have
$$
\pi\u=\u_T=(u_1,u_2,0),\qq \nu\cdot \u=u_3.
$$
Hence
$$
{\p u_3\over\p\nu}={\p\over\p\nu}(\nu\cdot\u)=\div\u-\div_\Gamma(\pi\u)-2(\nu\cdot\u)H(x).
$$
So
$$
u_1=0,\qq u_2=0,\qq {\p u_3\over \p\nu}=0\qq\text{on } \Gamma_{R,T}.
$$
We can write the equations for $u_1$ and $u_2$ as follows:
\eq\label{eq-u12}
\left\{\aligned
&\p_tu_j-a\Delta u_j+cu_j=f_j-h_j,\q&  (t,x)\in Q_{R,T},\\
&u_j=0,\q &(t,x)\in \Gamma_{R,T},\\
&u_j(0,x)=u_j^0,\q& x\in B_R^+,
\endaligned\right.
\eeq
$j=1,2$, and write the equation for $u_3$ as follows:
\eq\label{eq-u3}
\left\{\aligned
&\p_tu_3-a\Delta u_3+cu_3=f_3-h_3, \q&  (t,x)\in Q_{R,T},\\
&{\p u_3\over\p\nu}=0,\q & (t,x)\in \Gamma_{R,T},\\
&u_3(0,x)=u_3^0,\q&x\in B_R^+.
\endaligned\right.
\eeq

However, since $\Gamma_{R,T}$ is only a subset of the parabolic boundary of $Q_{R,T}$,
\eqref{eq-u12} and \eqref{eq-u3} are not exactly the standard initial-boundary value problem of parabolic equations with Dirichlet or Norman boundary condition.
We shall modify $u_j$'s to get the standard initial-boundary problems of parabolic equations.

Given $0<R<R_0$, we can take a domain $U$ with $C^{2+\a}$ boundary and a smooth function $\eta$ supported in $U$ such that
\eq\label{domain-U}
B_R^+\subset U\subset B_{R_0}^+,\q \eta(x)=1\;\;\text{if } x\in B_R^+,\q \eta(x)=0\;\;\text{if } x\in B^+_{R_0}\setminus\ol{U}.
\eeq
In fact we can first choose $\tilde R$ such that $R<\tilde R<R_0$. Take a smooth cur-off function $\eta$ such that
\eq\label{eta}
\spt(\eta)\subset B_{\tilde R},\q \eta=1\q\text{in }B_{\tilde R},\qq {\p\eta\over \p\nu}=0\q\text{on }\Sigma_{\tilde R}.
\eeq
Then we  take a domain $\tilde U$ with smooth boundary such that
$$
\overline{B}_{\tilde R}\subset \tilde U\subset \overline{\tilde U}\subset B_{R_0}.
$$
We can choose $\tilde U$ such that $U\equiv \tilde U\cap B_{R_0}^+$ has $C^{2+\a}$ boundary. Then $U$ and $\eta$ satisfy \eqref{domain-U}.

Let $\w=\eta\u$. Then
$$\w(t,x)=\0\q\text{for }0< t\leq T,\;\; x\in \p U\setminus \Sigma_{R_1}.
$$
Moreover we actually have $\w(t,x)=0$ for $x\in  U\setminus B_{R_1}$, thus
$$
{\p\over \p\nu}(\nu\cdot\w)=0\q\text{for }\;\; x\in \p U\setminus \Sigma_{R_1}.
$$
If $x\in \Sigma_{R_1}$, then from \eqref{eta} we have
$$
{\p \over \p\nu}(\nu\cdot\w)={\p\over\p\nu}(\eta\nu\cdot\u)=\eta{\p\over\p\nu}(\nu\cdot\u)+(\nu\cdot\u){\p\over \p\nu}\eta=0.
$$
So we have
\eq
\w_T=\0,\qq {\p\over\p\nu}(\nu\cdot\w)=0\qq\text{if } 0<t\leq T,\;\; x\in\p U.
\eeq

Denote
$$
G_T=(0,T]\times U,\qq
L_T=(0,T]\times \p U.
$$
We see that $\w$ is a weak solution of a modified system on $G_T$, namely
\eq\label{eqw}
\left\{\aligned
&{\p\w\over\p t}-a\,\Delta \w+\mB\,\curl\w+c\w=\F,\q\div\w=g,\q\;& (t,x)\in G_T,\\
&\w_T=\0,\qqq\, &(t,x)\in L_T,\\
&\w(0,x)=\w^0,\qq\;\;\;& x\in U,
\endaligned\right.
\eeq
where
$$\aligned
&\F=\eta\f-a(\Delta\eta\u+2\sum_{j=1}^3\p_j\eta\p_j\u)-\mB(\nabla\eta\times\u),\\
&g=\nabla\eta\cdot\u,\qq \w^0=\eta\u^0.
\endaligned
$$

Now we write
$$
\w=(w_1,w_2,w_3)^t,\q \F=(F_1,F_2,F_3)^t,\q
\mB\,\curl\w=\H=(H_1,H_2,H_3)^t.
$$
Then $w_1$, $w_2$ satisfy
\eq\label{eq-w12}
\left\{\aligned
&\p_tw_j-a\Delta w_j+cw_j=F_j-H_j,\q&  (t,x)\in G_{T},\\
&w_j=0,\q &(t,x)\in L_T,\\
&w_j(0,x)=w_j^0,\q& x\in U,
\endaligned\right.
\eeq
$j=1,2$, and $w_3$ satisfies
\eq\label{eq-w3}
\left\{\aligned
&\p_tw_3-a\Delta w_3+cw_3=F_3-H_3, \q&  (t,x)\in G_{T},\\
&{\p w_3\over\p\nu}=0,\q & (t,x)\in L_T,\\
&w_3(0,x)=w_3^0,\q&x\in U.
\endaligned\right.
\eeq
From the assumption on $\u^0$ and \eqref{adm2} we see that the following compatibility for the parabolic Dirichlet problem \eqref{eq-w12} condition is satisfied for $x\in \p U$ and for $j=1,2$:
$$
w^0_j(x)=0,\qq -a(0,x)\Delta w^0_j(x)=F_j(0,x)-H_j(0,x).
$$

We can apply the theory of regularity of parabolic equations to get a priori estimates of the solutions $\w$ in terms of $F_j$'s and $H_j$'s. However we can not directly get the final estimation by iteration the local estimates on $B_R^+$. Recall that the standard iteration processes  such as the bootstrap argument require the right hand terms be controlled by the unknowns. In our case $F_j$'s can be controlled by $\nabla \u$, but not by $\nabla \w$. So we can not improve the regularity on $F_j$'s over the whole region $B_R^+$ by iteration.
Nevertheless, we have improved the regularity of $\w$ in $B_R^+$, then we get the improved regularity of $\u$ in $B_{R_1}^+$ with some $R_1<R$ where $\eta=1$. Hence we can improve the regularity of $F_j$'s over $B_{R_1}^+$. Then we can iterate the above estimation to get the further improved estimates of the solution $\w$ on $B_{R_1}^+$ in terms of $F_j$'s.
We iterate this procedure in a finite times to get improved estimates on smaller regions.


Following the above idea, we shall first estimate the estimates of the solution $\w$ of \eqref{eqw} in terms of $\F$ and $g$.

(a) First of all, by the Sobolev imbedding in $\Bbb R^3$ we have
$$
L^2(0,T; H^2(U))\hookrightarrow L^2(0,T, W^{1,p}(U)),\q\forall 1<p<\infty.
$$

(b)
\eq\label{F}
\aligned
&|\H|\leq C|\curl\w|,\\
&|\F|\leq C(|\f|+|\u|+|\nabla \u|).
\endaligned
\eeq

In the following we derive the $L^p$ estimates.

{\it Step 1}.  $L^p$ estimate for $w_1, w_2$.

We take the following iteration argument. If $1<p<\infty$ is such that
\eq\label{dpsi-Lp}
\|\curl\w\|_{L^p(G_T)}< \infty,\qq \|\F\|_{L^p(G_T)}<\infty,
\eeq
then
\eq\label{H-Lp}
\|\H\|_{L^p(G_T)}<\infty,
\eeq
and we can
apply the global $L^p$ estimate for Dirichlet problem of heat equation (see \cite[p.176, Theorem 7.17]{Lib})
to \eqref{eq-w12} to get, for $j=1,2$,
\eq\label{wjLp1}
\aligned
\|w_j\|_{W^{2,1,p}(G_T)}\leq& C\{ \|F_j-H_j\|_{L^p(G_T)}+\|w_j^0\|_{W^{2,p}(U)}\}\\
\leq & C\{\|\F\|_{L^p(G_T)}+\|\curl\w\|_{L^p(G_T)}+\|w_j^0\|_{W^{2,p}(U)}\},
\endaligned
\eeq
where $C$ depends only on $U, T, a, c, p$.

Now \eqref{dpsi-Lp} is true for $p=2$ by the assumption, hence by \eqref{wjLp1} we have $$
w_j\in W^{2,1,2}(G_T),\q j=1,2.
$$
Then by Sobolev imbedding (see \cite[p.26, Theorem 3.14 (i)]{H} with
$$
p_1\equiv q={(n+2)p\over n+2-p}={5p\over 3}={10\over 3}
$$
when $p=2$) we see that $|\nabla_x w_j|\in L^{p_1}(G_T)$ with $p_1=10/3$, and
$$\aligned
\|\nabla_x w_j\|_{L^{p_1}(G_T)}\leq& C\|w_j\|_{W^{2,1,2}(G_T)}
\leq C\{\|\F\|_{L^2(G_T)}+\|\curl\w\|_{L^2(G_T)}+\|w_j^0\|_{W^{2,2}(U)}\}
\endaligned
$$
for $j=1, 2$, where $C$ depends only on $U, T, a, c, p_1$.

{\it Step 2}. $L^p$ estimate for $w_3$.

If $1<p<\infty$ is such that \eqref{dpsi-Lp} holds hence \eqref{H-Lp} is true, then we can apply the global $L^p$ estimate for Neumann problem of heat equation to \eqref{eq-w3} to get
\eq\label{w3Lp1}
\aligned
\|w_3\|_{W^{2,1,p}(G_T)}\leq& C\{\|F_3-H_3\|_{L^p(G_T)}   +\|w_3^0\|_{W^{2,p}(U)}\}\\
\leq &  C\{\|\F\|_{L^p(G_T)}+\|\curl\w\|_{L^p(G_T)}   +\|w_3^0\|_{W^{2,p}(U)}\},
\endaligned
\eeq
where $C$ depends only on $U, T, a, c, p$.

Now \eqref{dpsi-Lp} is true for $p=2$ by the assumption, hence by \eqref{w3Lp1} $w_3\in W^{2,1,2}(G_T)$. Then from the Sobolev imbedding (see \cite[p.26, Theorem 3.14 (i)]{H} with $p=2$) we see that $|\nabla_x w_3|\in L^{p_1}(G_T)$ with $p_1=10/3$, and
$$
\|\nabla_x w_3\|_{L^{p_1}(G_T)}\leq C\|w_3\|_{W^{2,1,2}(G_T)},
$$
where $C$ depends only on $U, T, a, c, p_1$.

Combining the results for $w_1, w_2, w_3$ we get
\eq\label{w-Lp1}
\aligned
&\|\nabla_x \w\|_{L^{p_1}(G_T)}\leq\|\w\|_{W^{2,1,p}(G_T)}
\leq C\{\|\F\|_{L^2(G_T)}+\|\curl\w\|_{L^2(G_T)}+\|\w^0\|_{W^{2,2}(U)}\},
\endaligned
\eeq
where $C$ depends only on $U, T, a, c, p_1$.
It follows that \eqref{dpsi-Lp} is true with $p$ replaced by $p_1=10/3$.

{\it Step 3.} By the choice of the cut-off function $\eta$ we have $\eta=1$ on $\overline{B^+_R}$, so $\w=\u$ on $\overline{B^+_R}$. From steps 1 and 2 we see that
$$
\u\in W^{2,1,2}(Q_{R,T},\Bbb R^3),
$$
and
\eq\label{u-Lp1}
\aligned
\|\u\|_{L^{p_1}(Q_{R,T})}\leq &\|\u\|_{W^{2,1,2}(Q_{R,T})}\leq \|\w\|_{W^{2,1,2}(G_T)}\\
\leq&  C\{\|\F\|_{L^{2}(G_T)}+\|\curl\w\|_{L^{2}(G_T)}+\|\w^0\|_{W^{2,2}(U)}\},
\endaligned
\eeq
where $C$ depends only on $R, R_0, T, U, \eta, a, c, p_1$. We can construct the domain $U$ in \eqref{domain-U} and the cut-off function $\eta$ which depend only on $R$ and $R_0$. Then in the above inequality the constant $C$ depends only on $R, R_0, T, a, c$.
Then we have
$$\aligned
&\|\curl\w\|_{L^2(G_T)}\leq C\{ \|\u\|_{L^2(G_T)}+\|\nabla_x \u\|_{L^2(G_T)}\},\\
&\|\w^0\|_{W^{2,2}(U)}\leq C\|\u^0\|_{W^{2,2}(B_{R_0}^+)},
\endaligned
$$
where $C$ depends only on $R, R_0$. From \eqref{F} and \eqref{u-Lp1} we see that, for $p_1=10/3$,
$$\aligned
\|\F\|_{L^{p_1}(Q_{R,T})}\leq& C\{\|\f\|_{L^{p_1}(Q_{R,T})}+\|\u\|_{L^{p_1}(Q_{R,T})}+\|\nabla \u\|_{L^{p_1}(Q_{R,T})}\}\\
\leq & C\{\|\f\|_{L^{p_1}(G_T)}+\|\F\|_{L^2(G_T)}+\|\curl\w\|_{L^2(G_T)}+\|\w^0\|_{W^{2,2}(U)}\}\\
\leq &C\{\|\f\|_{L^{p_1}(G_T)}
+\|\u\|_{L^2(G_T)}+\|\nabla_x\u\|_{L^2(G_T)}+\|\curl\w\|_{L^2(G_T)}+\|\w^0\|_{W^{2,2}(U)}\}\\
\leq &C\{\|\f\|_{L^{p_1}(Q_{R_0,T})}+\|\u\|_{L^2(Q_{R_0,T})}+\|\nabla_x\u\|_{L^2(Q_{R_0,T})}+\|\u^0\|_{W^{2,2}(B^+_{R_0})}\},
\endaligned
$$
where $C$ depends only on $R, R_0, T, a, c, p_1$.

{\it Step 4.}
Let $0<R_2<R_1<R_0$ be given. We change $R$ to $R_1$ in the above argument to get that $\u\in W^{2,1,2}(Q_{R_1,T})$ with
\eq\label{u-Lp1R1}
\aligned
&\|\u\|_{L^{p_1}(Q_{R_1,T})}\leq \|\u\|_{W^{2,1,2}(Q_{R_1,T})}\\
\leq& C\{\|\f\|_{L^{p_1}(Q_{R_0,T})}+\|\u\|_{L^2(Q_{R_0,T})}+\|\nabla_x\u\|_{L^2(Q_{R_0,T})}+\|\u^0\|_{W^{2,2}(B^+_{R_0})}\},
\endaligned
\eeq
where $p_1=10/3$ and $C$ depends only on $R_1, R_0, T, a, c, p_1$. Then
$$
\|\F\|_{L^{p_1}(Q_{R_1,T})}\leq C\{\|\f\|_{L^{p_1}(\O_{R_0,T})}+\|\u\|_{L^2(Q_{R_0,T})}+\|\nabla_x\u\|_{L^2(Q_{R_0,T})}+\|\u^0\|_{W^{2,2}(B^+_{R_0})}\},
$$
where $C$ depends only on $R_1, R_0, T, a, c, p_1$. From \eqref{w-Lp1} and \eqref{u-Lp1} we have
\eq\label{H-Lp1}
\aligned
&\|\H\|_{L^{p_1}(Q_{R_1,T})}\leq C\|\curl\w\|_{L^{p_1}(Q_{R_1,T})}\\
\leq& C\{\|\f\|_{L^{p_1}(Q_{R_0,T})}+\|\u\|_{L^2(Q_{R_0,T})}+\|\nabla_x\u\|_{L^2(Q_{R_0,T})}+\|\u^0\|_{W^{2,2}(B^+_{R_0})}\},
\endaligned
\eeq
where $C$ depends only on $R_1, R_0, T, a, c, p_1$.

Then we can repeat the above argument with $R$ and $R_0$ replaced by $R_2$ and $R_1$ to get $W^{2,1,p_1}$ estimate on $\O_{R_2,T}$. More precisely,
we can take a domain $U_2$ with $C^{2+\a}$ boundary and a smooth function $\eta_2$ supported in $U_2$ such that
\eq\label{domain-U2}
B_{R_2}^+\subset U_2\subset B_{R_1}^+,\q \eta_2(x)=1\;\;\text{if } x\in B_{R_2}^+,\q \eta_2(x)=0\;\;\text{if } x\in B^+_{R_1}\setminus U_2.
\eeq
Set $G^2_T=(0,T]\times U_2$.
Then the conditions \eqref{dpsi-Lp} and \eqref{H-Lp} hold on $G^2_T$ with $p$ replaced by $p_1$.
As in steps 1.1 and 1.2, we apply the $L^p$ estimates of heat equation to get \eqref{wjLp1} and \eqref{w3Lp1}
on $G^2_T$ with $p$ replaced by $p_1$, then we get \eqref{u-Lp1} with $R$ and $R_0$  replaced
by $R_2$ and $R_1$. So we have now
\eq\label{u-Lp2R2}
\aligned
&\|\u\|_{L^{p_2}(Q_{R_2,T})}\leq \|\u\|_{W^{2,1,p_1}(Q_{R_2,T})}\\
\leq& C\{\|\f\|_{L^{p_1}(Q_{R_1,T})}+\|\u\|_{L^{p_1}(Q_{R_1,T})}
+\|\nabla\u\|_{L^{p_1}(Q_{R_1,T})}
+\|\u^0\|_{W^{2,p_1}(B^+_{R_1})}\}\\
\leq& C\{\|\f\|_{L^{p_1}(Q_{R_0,T})}+\|\u\|_{L^2(Q_{R_0,T})}+\|\nabla_x\u\|_{L^2(Q_{R_0,T})}+\|\u^0\|_{W^{2,p_1}(B^+_{R_0})}\},
\endaligned
\eeq
where
$$
p_2={(n+2)p_1\over n+2-p_1}={5p_1\over 5-p_1}=10
$$
and $C$ depends only on $R_1, R_0, T, a, c, p_2$. It further follows that
$$
\F\in L^{p_2}(Q_{R_2,T},\Bbb R^3),\qq \H\in L^{p_2}(Q_{R_2,T},\Bbb R^3).
$$

{\it Step 5}.
Now we fix $0<R<R_0$ and let
$$
R_k=R_0-{R_0-R\over 2^k},\q p_{k+1}={5p_k\over 5-p_k},\q k=1, 2,3,\cdots.
$$
After having got
$$
\u\in W^{2,1,p_k}(Q_{R_k,T},\Bbb R^3),
$$
hence
$$
\F\in L^{p_{k+1}}(Q_{R_k,T},\Bbb R^3),\qq \H\in L^{p_{k+1}}(Q_{R_k,T},\Bbb R^3),
$$
we apply the iteration argument with $p=p_{k+1}$ on $Q_{R_k,T}$ to conclude that
$$
\u\in W^{2,1,p_{k+1}}(Q_{R_{k+1,T}},\Bbb R^3),
$$
and
\eq
\aligned
&\|\u\|_{W^{2,1,p_{k+1}}(Q_{R_{k+1},T})}\\
\leq& C\{\|\f\|_{L^{p_{k+1}}(Q_{R_k,T})}+\|\u\|_{L^{p_k}(Q_{R_k,T})}+\|\nabla \u\|_{L^{p_k}(Q_{R_k,T})}+\|\u^0\|_{W^{2,p_{k+1}}(B^+_{R_k})}\}\\
\leq& C\{\|\f\|_{L^{p_{k+1}}(Q_{R_0,T})}+\|\u\|_{L^2(Q_{R_0,T})}+\|\nabla_x\u\|_{L^2(Q_{R_0,T})}+\|\u^0\|_{W^{2,p_{k+1}}(B^+_{R_0})}\},
\endaligned
\eeq
$C$ depends only on $R, R_0, T, a, c, p_k$.

Note that $p_k\to\infty$ and $0<R<R_k<R_0$ for all $k$.

{\it Step 6.}
For any $1<q<\infty$, we can choose $U$ and $\eta$ such that \eqref{domain-U} and \eqref{eta} hold. Then we repeat the $L^q$ estimate on $\w=\eta\u$ on $G_T=(0,T]\times U$ to get $w^{2,1,q}$ estimate for $\w$ on $G_T$, which implies that
$$
\u\in W^{2,1,q}(Q_{R,T},\Bbb R^3),
$$
and
\eq
\aligned
&\|\u\|_{W^{2,1,q}(Q_{R,T})}\leq \|\w\|_{W^{2,1,q}(G_T)}\\
\leq& C\{\|\f\|_{L^q(G_T)}+\|\w\|_{L^q(G_T)}+\|\curl\w\|_{L^q(G_T)}+\|\w^0\|_{W^{2,q}(U)}\}\\
\leq & C\{ \|\f\|_{L^q(Q_{R_0,T})}+\|\u\|_{L^2(Q_{R_0,T})}+\|\nabla_x\u\|_{L^2(Q_{R_0,T})}+\|\u^0\|_{W^{2,q}(B^+_{R_0})}\},
\endaligned
\eeq
$C$ depends only on $R, R_0, T, a, c, q$.
\end{proof}

\subsection{$C^{1+\a,(1+\a)/2}$-estimates}\

\begin{Cor}\label{Cor2.2} Under the assumption of Lemma \ref{Lem2.1}, if $0<\a<1$ and $0<R<R_0$, the weak solution $\u$ is of
$C^{1+\a,\a/2}$, and
$$
\|\u\|_{C^{1+\a,(1+\a)/2}(\overline{Q}_{R,T})}\leq C\{ \|\f\|_{L^q(Q_{R_0,T})}+\|\u\|_{L^2(Q_{R_0,T})}+\|\nabla_x\u\|_{L^2(Q_{R_0,T})}+\|\u^0\|_{W^{2,q}(B^+_{R_0})}\},
$$
where $5/(1-\a)<q<\infty$, $C$  depends only on $R, R_0, T, a, c, \a, q$.
\end{Cor}

\begin{proof} We apply the Sobolev imbedding  given in \cite[p.26, Theorem 3.14 (3)]{H} (with $p=q>5/(1-\a)$) to conclusion that
$$
\u\in C^{1+\a,(1+\a)/2}(\overline{Q}_{R,T},\Bbb R^3)
$$
and
$$
\|\u\|_{C^{1+\a,(1+\a)/2}(\overline{Q}_{R,T})}\leq C\|\u\|_{W^{2,1,q}(Q_{R,T})}.
$$
Then the conclusion follows from Lemma \ref{Lem2.1}.
\end{proof}

\subsection{Schauder estimates}\

\begin{Lem}\label{Lem2.3}
Assume that
$$\aligned
&a,\; c\in C^{\a,\a/2}(\overline{Q}_{R_0,T}),\qq a(t,x)\geq a_0>0,\\
&\mB\in C^{\a,\a/2}(\overline{Q}_{R_0,T},M(3)),\qq \f\in C^{\a,\a/2}(\overline{Q}_{R_0,T},\Bbb R^3),\\
& \u^0\in C^{2+\a}(\overline{B}_{R_0}^+,\Bbb R^3),\qq \div\u^0=0\;\;\text{\rm in }B_{R_0},\qq \u^0_T=\0\;\;\text{\rm on }\Sigma_{R_0},
\endaligned
$$
and assume \eqref{adm} holds.
If $\u$ is a weak solution of \eqref{eqB.1} on $Q_{R_0,T}$, then for any $0<R<R_0$ we have $\u\in C^{2+\a,1+\a/2}(\overline{Q}_{R,T})$ and
$$
\|\u\|_{C^{2+\a,1+\a/2}(\overline{Q}_{R,T})}\leq C\{\|\f\|_{C^{\a,\a/2}(\overline{Q}_{R_0,T})}+\|\u\|_{L^2(Q_{R_0,T})}+\|\nabla_x\u\|_{L^2(Q_{R_0,T})}+\|\u^0\|_{C^{2+\a}(\overline{B}_{R_0}^+)}\},
$$
where $C$ depends only on $R_0$, $R$, $T$, $\a$, $\mB$, $a$, $c$.
\end{Lem}

\begin{proof}
In the following we use the fact that if $u\in  C^{1+\a,(1+\a)/2}(\overline{\O}_{R,T})$, then $\nabla_x u\in C^{\a,\a/2}(\overline{\O}_{R,T})$.

Take $R_1$ depending only on $R$ and $R_0$ such that $R<R_1<R_0$.
From Corollary \ref{Cor2.2} we know that $\u\in C^{1+\a,(1+\a)/2}(\overline{Q}_{R_1,T},\Bbb R^3)$, hence $\curl\u\in C^{\a,\a/2}(\overline{Q}_{R_1,T},\Bbb R^3)$, thus
$$
\mB\,\curl\u\in C^{\a,\a/2}(\overline{Q}_{R_1,T},\Bbb R^3).
$$

Take a domain $U$ and a smooth cut-off function $\eta$ such that \eqref{domain-U} is satisfied with $R_0$ replaced by $R_1$. Set $G_T=(0,T]\times U$ and $L_T=(0,T]\times\p U$. Set $\w=\eta\u$. Then $\w$ is a weak solution of \eqref{eqw}, where $\F\in C^{\a,\a/2}(\overline{G}_T,\Bbb R^3)$ because $\f\in C^{\a,\a/2}(\overline{G}_T,\Bbb R^3)$ and $\p_j\u\in C^{\a,\a/2}(\overline{G}_T,\Bbb R^3)$.

Again we write $\w=(w_1,w_2,w_3)^t$, $\F=(F_1,F_2,F_3)^t$, $\mB\,\curl\w=\H=(H_1,H_2,H_3)^t$. Then $w_1, w_2$ are weak solutions of \eqref{eq-w12} and $w_3$ is a weak solution of \eqref{eq-w3}.

Applying the global Schauder estimate for Dirichlet problem (see \cite[p.78, Theorem 4.28]{Lib})
to \eqref{eq-w12} we have
$$
\|w_j\|_{C^{2+\a,1+\a/2}(\overline{G}_T)}\leq C_D\{\|F_j-H_j\|_{C^{\a,\a/2}(\overline{G}_T)}+\|w_j^0\|_{C^{2+\a}(\overline{U})}\},\q j=1,2.
$$
Then applying the global Schauder estimate for Neumann problem \cite[p.79, Theorem 4.31]{Lib}
to \eqref{eq-w3} we have
$$
\|w_3\|_{C^{2+\a,1+\a/2}(\overline{G}_T)}\leq C_N\{\|F_3-H_3\|_{C^{\a,\a/2}(\overline{G}_T)}+\|w_3^0\|_{C^{2+\a}(\overline{U})}\}.
$$
Therefore $\w\in C^{2+\a,1+\a/2}(\overline{G}_T,\Bbb R^3)$, and
\eq\label{wca}
\|\w\|_{C^{2+\a,1+\a/2}(\overline{G}_T)}\leq C\{\|\F\|_{C^{\a,\a/2}(\overline{G}_T)}+\|\H\|_{C^{\a,\a/2}(\overline{G}_T)}+\|\w^0\|_{C^{2+\a}(\overline{U})}\}.
\eeq

Note that
$$\aligned
\|\F\|_{C^{\a,\a/2}(\overline{G}_T)}\leq & C\{\|\f\|_{C^{\a,\a/2}(\overline{G}_T)}+\|\u\|_{C^{\a,\a/2}(\overline{G}_T)}+\|\nabla_x\u\|_{C^{\a,\a/2}(\overline{G}_T)},
\\
\|\H\|_{C^{\a,\a/2}(\overline{G}_T)}\leq& C\|\curl\w\|_{C^{\a,\a/2}(\overline{G}_T)}
\leq C\|\w\|_{C^{1+\a,(1+\a)/2}(\overline{G}_T)}.
\endaligned
$$
From these and \eqref{wca} we get
$$
\|\w\|_{C^{2+\a,1+\a/2}(\overline{G}_T)}\leq C\{\|\f\|_{C^{\a,\a/2}(\overline{G}_T)}+\|\u\|_{C^{1+\a,(1+\a)/2}(\overline{G}_T)}+\|\w^0\|_{C^{2+\a}(\overline{U})}\}.
$$
Using the construction of $G_T$ and Corollary \ref{Cor2.2} with $R$ replaced by $R_1$, and with the index $q$ determined by $\a$, we have
$$\aligned
&\|\u\|_{C^{1+\a,(1+\a)/2}(\overline{G}_T)}\leq \|\u\|_{C^{1+\a,(1+\a)/2}(\overline{Q}_{R_1,T})}\\
\leq& C\{\|\f\|_{L^q(Q_{R_1,T})}+\|\u\|_{L^2(Q_{R_0,T})}+\|\nabla_x\u\|_{L^2(Q_{R_0,T})}
+\|\u^0\|_{C^{2+\a}(\overline{B}_{R_0}^+)}\}\\
\leq& C\{\|\f\|_{C^{\a,\a/2}(\overline{Q}_{R_0,T})}+\|\u\|_{L^2(Q_{R_0,T})}+\|\nabla_x\u\|_{L^2(Q_{R_0,T})}
+\|\u^0\|_{C^{2+\a}(\overline{B}_{R_0}^+)}\},
\endaligned
$$
where $C$ depends only on $R, R_0, T, \a, \mB, a, c$, as $R_1$ is determined by $R$ and $R_0$.
Hence we get
$$
\|\u\|_{C^{2+\a,1+\a/2}(\overline{Q}_{R,T})}\leq C\{\|\f\|_{C^{\a,\a/2}(\overline{Q}_{R_0,T})}+\|\u\|_{L^2(Q_{R_0,T})}+\|\nabla_x\u\|_{L^2(Q_{R_0,T})}+\|\u^0\|_{C^{2+\a}(\overline{B}^+_{R_0})}\}.
$$
\end{proof}

\section{Estimates Near Curved Boundary}

\subsection{Computations in local coordinates near
boundary}\

Let us briefly recall the local coordinates near boundary $\p\O$
determined by a diffeomorphism that straightens a piece of
surface, see \cite[section 3]{P} and \cite[Appendix]{BaP}. Let us fix a point $x_0\in \p\O$, and
introduce new variables $y_1$, $y_2$ such that $\p\O$ can be
represented (at least near $x_0$) by $\r=\r(y_1,y_2)$, and
$\r(0,0)=x_0$. Here and henceforth we denote $y=(y_1, y_2) $ and use
the notation $\r_j(y)=\p_{y_j}\r(y)$, $\r_{ij}=\p_{y_iy_j}\r(y)$,
etc. Let
$$\n(y)={\r_1(y)\times\r_2(y)\over |\r_1(y)\times\r_2(y)|}.
$$
We choose $(y_1, y_2)$ in such a way that $\n(y)$ is the inward
normal of $\p\O$, and that the $y_1$- and $y_2$-curves on $\p\O$
are the lines of principal curvature; thus, $\r_1(y)$ and $\r_2(y)$ are
orthogonal to each other. Let
$$g_{ij}(y)=\r_i(y)\cdot\r_j(y),\qq
g(y)=\det (g_{ij}(y))=g_{11}(y)g_{22}(y).
$$
Let us define a map $\mathcal F$ by
$$x={\mathcal F}(y, z)=\r(y_1,
y_2)+z\n(y_1, y_2).
$$
${\mathcal F}$ is a diffeomorphism from a ball
$B_R(0)$ onto a neighborhood $\mathcal U$ of the point $x_0$,  and it
maps the half ball $B_R^+(0)$ onto a subdomain $\mathcal U\cap\O$, and
maps the disc $\{(y_1,y_2,0): y_1^2+y_2^2<R^2\}$ onto a subset $\Gamma$ of
$\p\O$.

Denote the partial derivative $\p_{y_j}$ by $\p_j$, $j=1,2$, and denote $\p_z$ by $\p_3$.  Let
$$G_{ij}(y,z)=\p_i\mathcal F\cdot\p_j\mathcal F,\q i,j=1,2,3,
$$
and let
$G^{ij}(y,z)$ denote the elements of the inverse of the matrix
$(G_{ij}(y,z))_{3\times 3}$. Then
$$\aligned
&G_{jj}(y,z)=g_{jj}(y)[1-\k_j(y)z]^2,\qq G^{jj}={1\over G_{jj}},\qq j=1,2,\\
&G_{12}=G_{13}=G_{23}=G^{12}=G^{13}=G^{23}=0,\qq G_{33}=G^{33}=1,\\
&G(y,z)\equiv \text{det}(G_{ij}(y,z))=G_{11}(y,z)G_{22}(y,z).
\endaligned
$$
Note that if $\p\O$ is of $C^{k+\a}$, then $G_{jj}$ and $G^{jj}$ are of $C^{k-1+\a}$.

On the domain $\mathcal U$ we have an orthogonal coordinate framework
$\{\E_1,\E_2,\E_3\}$, where
$$\aligned
&\E_j(y,z)={\p_j\mathcal F\over|\p_j\mathcal
F|}={\r_j(y)\over\sqrt{g_{jj}}},\q j=1,2;\qq \E_3(y,z)={\p_3\mathcal
F\over|\p_3\mathcal F|}=\n(y).
\endaligned
$$

Given a vector field $\B$ defined on $\overline{\O}$, we can represent $\B$ in a neighborhood of $x_0\in\p\O$ in the
new variables $(y,z)\in B_R^+=\mathcal F^{-1}(\mathcal U\cap\O)$ as
follows:
\eq\label{A.1}
\aligned
&\tilde\B(y,z)=\B({\mathcal F}(y,z))
=\sum_{j=1}^3G^{jj}(y,z)b_j(y,z)\p_j{\mathcal F}(y,z)
=\sum_{j=1}^3\tilde B_j(y,z)\E_j(y,z),\\
&b_j(y,z)=\B(\mathcal F(y,z))\cdot\p_j{\mathcal F}(y,z),\qq \tilde
B_j(y,z)={b_j(y,z)\over\sqrt{G_{jj}(y,z)}}.
\endaligned
\eeq
We compute, at the point $x=\mathcal F(y,z)$,
\eq\label{A.2}
\aligned
&\curl\B(x)=\sum_{j=1}^3\tilde R_j(y,z)\E_j(y,z),\\
&\div\B={1\over\sqrt{G}}\Big[\sum_{j=1}^2\p_j\Big(\sqrt{G\over G_{jj}}\tilde B_j\Big)
+\p_3\Big(\sqrt{G\over G_{33}}\tilde B_3\Big)\Big],
\endaligned
\eeq
where
$$\aligned
\tilde R_1(y,z)&={1\over\sqrt{G_{22}G_{33}}}[\p_2(\tilde
B_3\sqrt{G_{33}})-\p_3(\tilde B_2\sqrt{G_{22}})]
={1\over \sqrt{G_{22}}}[\p_2b_3-\p_3b_2],\\
\tilde R_2(y,z)&={1\over\sqrt{G_{33}G_{11}}}[\p_3(\tilde
B_1\sqrt{G_{11}})-\p_1(\tilde B_3\sqrt{G_{33}})]
={1\over \sqrt{G_{11}}}[\p_3b_1-\p_1b_3],\\
\tilde R_3(y,z)&={1\over\sqrt{G_{11}G_{22}}}[\p_1(\tilde
B_2\sqrt{G_{22}})-\p_2(\tilde B_1\sqrt{G_{11}})]  ={1\over
\sqrt{G_{11}G_{22}}}[\p_1b_2-\p_2b_1].
\endaligned
$$

If we write
$$
\curl^2\B=\sum_{j=1}^3\tilde T_j(y,z)\E_j(y,z),
$$
then
$$\aligned
\tilde T_1(y,z)&={1\over\sqrt{G_{22}G_{33}}}[\p_2(\sqrt{G_{33}}\tilde
R_3)-\p_3(\sqrt{G_{22}}\tilde R_2)]\\
&={1\over\sqrt{G_{22}G_{33}}}\p_2\Big\{  {\sqrt{G_{33}}\over\sqrt{G_{11}G_{22}}}\Big[\p_1(\tilde
B_2\sqrt{G_{22}})-\p_2(\tilde B_1\sqrt{G_{11}})\Big]\Big\}\\
&- {1\over\sqrt{G_{22}G_{33}}}    \p_3\Big\{  {\sqrt{G_{22}}\over\sqrt{G_{33}G_{11}}}\Big[\p_3(\tilde
B_1\sqrt{G_{11}})-\p_1(\tilde B_3\sqrt{G_{33}})\Big]\Big\},
\endaligned
$$
$$\aligned
\tilde T_2(y,z)&={1\over\sqrt{G_{33}G_{11}}}[\p_3(\sqrt{G_{11}}\tilde
R_1)-\p_1(\sqrt{G_{33}}\tilde R_3)]\\
&={1\over\sqrt{G_{33}G_{11}}}\p_3\Big\{  {\sqrt{G_{11}}\over\sqrt{G_{22}G_{33}}}\Big[\p_2(\tilde
B_3\sqrt{G_{33}})-\p_3(\tilde B_2\sqrt{G_{22}})\Big]\Big\}\\
&- {1\over\sqrt{G_{33}G_{11}}}    \p_1\Big\{  {\sqrt{G_{33}}\over\sqrt{G_{11}G_{22}}}\Big[\p_1(\tilde
B_2\sqrt{G_{22}})-\p_2(\tilde B_1\sqrt{G_{11}})\Big]\Big\},
\endaligned
$$
$$\aligned
\tilde T_3(y,z)&={1\over\sqrt{G_{11}G_{22}}}[\p_1(\sqrt{G_{22}}\tilde
R_2)-\p_2(\sqrt{G_{11}}\tilde R_1)]\\
&={1\over\sqrt{G_{11}G_{22}}}\p_1\Big\{  {\sqrt{G_{22}}\over\sqrt{G_{33}G_{11}}}\Big[\p_3(\tilde
B_1\sqrt{G_{11}})-\p_1(\tilde B_3\sqrt{G_{33}})\Big]\Big\}\\
&- {1\over\sqrt{G_{11}G_{22}}}    \p_2\Big\{  {\sqrt{G_{11}}\over\sqrt{G_{22}G_{33}}}\Big[\p_2(\tilde
B_3\sqrt{G_{33}})-\p_3(\tilde B_2\sqrt{G_{22}})\Big]\Big\}.
\endaligned
$$

Let $\u$ be a solution of \eqref{eqB}. In the neighbourhood $\mU$ near boundary, we write
\eq
\aligned
&\tilde\u(t,y,z)=\u(t,{\mathcal F}(y,z))
=\sum_{j=1}^3\tilde u_j(t,y,z)\E_j(y,z),\\
&\div\u={1\over\sqrt{G}}\Big[\sum_{j=1}^2\p_j\Big(\sqrt{G\over G_{jj}}\tilde u_j\Big)
+\p_3\Big(\sqrt{G\over G_{33}}\tilde u_3\Big)\Big],\\
&\curl\u(x)=\sum_{j=1}^3\tilde R_j(t,y,z)\E_j(y,z),\qq
\curl^2\u=\sum_{j=1}^3\tilde T_j(t,y,z)\E_j(y,z),\\
&\mB\,\curl\u=\h=\sum_{j=1}^3\tilde h_j(t,y,z)\E_j(y,\z),\qq \f=\sum_{j=1}^3\tilde f_j(t,y,z)\E_j(y,\z).
\endaligned
\eeq
Recall that $G_{33}=1$. We have
$$\aligned
\tilde T_1&={1\over\sqrt{G_{22}G_{33}}}\p_2\Big\{  {\sqrt{G_{33}}\over\sqrt{G_{11}G_{22}}}\Big[\p_1(\tilde
u_2\sqrt{G_{22}})-\p_2(\tilde u_1\sqrt{G_{11}})\Big]\Big\}\\
&- {1\over\sqrt{G_{22}G_{33}}}    \p_3\Big\{  {\sqrt{G_{22}}\over\sqrt{G_{33}G_{11}}}\Big[\p_3(\tilde
u_1\sqrt{G_{11}})-\p_1(\tilde u_3\sqrt{G_{33}})\Big]\Big\}\\
=&{1\over G_{22}\sqrt{G_{11}}}\Big[\p_{12}(\tilde
u_2\sqrt{G_{22}})-\p_{22}(\tilde u_1\sqrt{G_{11}})\Big]\\
&+{1\over\sqrt{G_{22}G_{33}}}\p_2({\sqrt{G_{33}}\over\sqrt{G_{11}G_{22}}})\Big[\p_1(\tilde
u_2\sqrt{G_{22}})-\p_2(\tilde u_1\sqrt{G_{11}})\Big]\\
&- {1\over G_{33}\sqrt{G_{11}}}\Big[\p_{33}(\tilde
u_1\sqrt{G_{11}})-\p_{13}(\tilde u_3\sqrt{G_{33}})\Big]\\
&- {1\over\sqrt{G_{22}G_{33}}}    \p_3( {\sqrt{G_{22}}\over\sqrt{G_{33}G_{11}}})\Big[\p_3(\tilde
u_1\sqrt{G_{11}})-\p_1(\tilde u_3\sqrt{G_{33}})\Big]\\
=&{1\over G_{22}\sqrt{G_{11}}}\Big[\sqrt{G_{22}}\p_{12}\tilde u_2+\p_2\tilde u_2\p_1\sqrt{G_{22}}+\p_1(\tilde u_2\p_2\sqrt{G_{22}})\\
&\qqq -\sqrt{G_{11}}\p_{22}\tilde u_1-\p_2\tilde u_1\p_2\sqrt{G_{11}}-\p_2(\tilde u_1\p_2\sqrt{G_{11}})\Big]\\
&+{1\over\sqrt{G_{22}G_{33}}}\p_2({\sqrt{G_{33}}\over\sqrt{G_{11}G_{22}}})\Big[\p_1(\tilde
u_2\sqrt{G_{22}})-\p_2(\tilde u_1\sqrt{G_{11}})\Big]\\
&- {1\over G_{33}\sqrt{G_{11}}}\Big[\sqrt{G_{11}}\p_{33}\tilde u_1+\p_3\tilde u_1\p_3\sqrt{G_{11}}+\p_3(\tilde u_1\p_3\sqrt{G_{11}}) -\p_{13}\tilde u_3\Big]\\
&- {1\over\sqrt{G_{22}G_{33}}}    \p_3( {\sqrt{G_{22}}\over\sqrt{G_{33}G_{11}}})\Big[\p_3(\tilde
u_1\sqrt{G_{11}})-\p_1(\tilde u_3\sqrt{G_{33}})\Big]\\
=&-{1\over G_{22}}\p_{22}\tilde u_1-\p_{33}\tilde u_1+{1\over\sqrt{G}}\p_{12}\tilde u_2+{1\over\sqrt{G_{11}}}\p_{13}\tilde u_3+\phi_1,
\endaligned
$$
where
$$\aligned
\phi_1=&{1\over G_{22}\sqrt{G_{11}}}\Big[\p_2\tilde u_2\p_1\sqrt{G_{22}}+\p_1(\tilde u_2\p_2\sqrt{G_{22}})
 -\p_2\tilde u_1\p_2\sqrt{G_{11}}-\p_2(\tilde u_1\p_2\sqrt{G_{11}})\Big]\\
&+{1\over\sqrt{G_{22}}}\p_2({1\over\sqrt{G_{11}G_{22}}})\Big[\p_1(\tilde
u_2\sqrt{G_{22}})-\p_2(\tilde u_1\sqrt{G_{11}})\Big]\\
&- {1\over \sqrt{G_{11}}}\Big[\p_3\tilde u_1\p_3\sqrt{G_{11}}+\p_3(\tilde u_1\p_3\sqrt{G_{11}})\Big]
- {1\over\sqrt{G_{22}}}    \p_3( {\sqrt{G_{22}}\over\sqrt{G_{11}}})\Big[\p_3(\tilde
u_1\sqrt{G_{11}})-\p_1\tilde u_3\Big].
\endaligned
$$

$$\aligned
\tilde T_2&={1\over\sqrt{G_{33}G_{11}}}\p_3\Big\{  {\sqrt{G_{11}}\over\sqrt{G_{22}G_{33}}}\Big[\p_2(\tilde
u_3\sqrt{G_{33}})-\p_3(\tilde u_2\sqrt{G_{22}})\Big]\Big\}\\
&- {1\over\sqrt{G_{33}G_{11}}}    \p_1\Big\{  {\sqrt{G_{33}}\over\sqrt{G_{11}G_{22}}}\Big[\p_1(\tilde
u_2\sqrt{G_{22}})-\p_2(\tilde u_1\sqrt{G_{11}})\Big]\Big\}\\
=&{1\over G_{33}\sqrt{G_{22}}}\Big[\p_{23}(\tilde
u_3\sqrt{G_{33}})-\p_{33}(\tilde u_2\sqrt{G_{22}})\Big]\\
&+{1\over\sqrt{G_{33}G_{11}}}\p_3({\sqrt{G_{11}}\over\sqrt{G_{22}G_{33}}})\Big[\p_2(\tilde
u_3\sqrt{G_{33}})-\p_3(\tilde u_2\sqrt{G_{22}})\Big]\\
&- {1\over G_{11}\sqrt{G_{22}}}\Big[\p_{11}(\tilde
u_2\sqrt{G_{22}})-\p_{12}(\tilde u_1\sqrt{G_{11}})\Big]\\
&- {1\over\sqrt{G_{33}G_{11}}}    \p_1( {\sqrt{G_{33}}\over\sqrt{G_{11}G_{22}}})\Big[\p_1(\tilde
u_2\sqrt{G_{22}})-\p_2(\tilde u_1\sqrt{G_{11}})\Big]\\
=&{1\over G_{33}\sqrt{G_{22}}}\Big[\p_{23}\tilde u_3-\sqrt{G_{22}}\p_{33}\tilde u_2 -\p_3\tilde u_2\p_3\sqrt{G_{22}}-\p_3(\tilde u_2\p_3\sqrt{G_{22}})\Big]\\
&+{1\over\sqrt{G_{33}G_{11}}}\p_3({\sqrt{G_{11}}\over\sqrt{G_{22}G_{33}}})\Big[\p_2(\tilde
u_3\sqrt{G_{33}})-\p_3(\tilde u_2\sqrt{G_{22}})\Big]\\
&- {1\over G_{11}\sqrt{G_{22}}}\Big[\sqrt{G_{22}}\p_{11}\tilde u_2+\p_1\tilde u_2\p_1\sqrt{G_{22}}+\p_1(\tilde u_2\p_1\sqrt{G_{22}})\\
&\qqq  - \sqrt{G_{11}}\p_{12}\tilde u_1-\p_2\tilde u_1\p_1\sqrt{G_{11}}-\p_1(\tilde u_1\p_2\sqrt{G_{11}}) \Big]\\
&- {1\over\sqrt{G_{33}G_{11}}}    \p_1( {\sqrt{G_{33}}\over\sqrt{G_{11}G_{22}}})\Big[\p_1(\tilde
u_2\sqrt{G_{22}})-\p_2(\tilde u_1\sqrt{G_{11}})\Big]\\
=&-{1\over G_{11}}\p_{11}\tilde u_2-\p_{33}\tilde u_2+{1\over\sqrt{G}}\p_{12}\tilde u_1+{1\over\sqrt{G_{22}}}\p_{23}\tilde u_3+\phi_2,
\endaligned
$$
where
$$\aligned
\phi_2=&-{1\over \sqrt{G_{22}}}\Big[\p_3\tilde u_2\p_3\sqrt{G_{22}}+\p_3(\tilde u_2\p_3\sqrt{G_{22}})\Big]
+{1\over\sqrt{G_{11}}}\p_3({\sqrt{G_{11}}\over\sqrt{G_{22}}})\Big[\p_2\tilde
u_3-\p_3(\tilde u_2\sqrt{G_{22}})\Big]\\
&- {1\over G_{11}\sqrt{G_{22}}}\Big[\p_1\tilde u_2\p_1\sqrt{G_{22}}+\p_1(\tilde u_2\p_1\sqrt{G_{22}})
-\p_2\tilde u_2\p_1\sqrt{G_{11}}-\p_1(\tilde u_1\p_2\sqrt{G_{11}})\Big]\\
&- {1\over\sqrt{G_{11}}}    \p_1( {1\over\sqrt{G}})\Big[\p_1(\tilde
u_2\sqrt{G_{22}})-\p_2(\tilde u_1\sqrt{G_{11}}\Big].
\endaligned
$$

$$\aligned
\tilde T_3&={1\over\sqrt{G_{11}G_{22}}}\p_1\Big\{  {\sqrt{G_{22}}\over\sqrt{G_{33}G_{11}}}\Big[\p_3(\tilde
u_1\sqrt{G_{11}})-\p_1(\tilde u_3\sqrt{G_{33}})\Big]\Big\}\\
&- {1\over\sqrt{G_{11}G_{22}}}    \p_2\Big\{  {\sqrt{G_{11}}\over\sqrt{G_{22}G_{33}}}\Big[\p_2(\tilde
u_3\sqrt{G_{33}})-\p_3(\tilde u_2\sqrt{G_{22}})\Big]\Big\}\\
=&{1\over G_{11}\sqrt{G_{33}}}\Big[\p_{13}(\tilde
u_1\sqrt{G_{11}})-\p_{11}(\tilde u_3\sqrt{G_{33}})\Big]\\
&+{1\over\sqrt{G_{11}G_{22}}}\p_1({\sqrt{G_{22}}\over\sqrt{G_{33}G_{11}}})\Big[\p_3(\tilde
u_1\sqrt{G_{11}})-\p_1(\tilde u_3\sqrt{G_{33}})\Big]\\
&- {1\over G_{22}\sqrt{G_{33}}}\Big[\p_{22}(\tilde
u_3\sqrt{G_{33}})-\p_{23}(\tilde u_2\sqrt{G_{22}})\Big]\\
&- {1\over\sqrt{G_{11}G_{22}}}    \p_2( {\sqrt{G_{11}}\over\sqrt{G_{22}G_{33}}})\Big[\p_2(\tilde
u_3\sqrt{G_{33}})-\p_3(\tilde u_2\sqrt{G_{22}})\Big]\\
=&{1\over G_{11}\sqrt{G_{33}}}\Big[\p_{13}\tilde u_1\sqrt{G_{11}}+\p_1\tilde u_1\p_3\sqrt{G_{11}}+\p_3(\tilde u_1\p_1\sqrt{G_{11}})-\p_{11}\tilde u_3 \Big]\\
&+{1\over\sqrt{G_{11}G_{22}}}\p_1({\sqrt{G_{22}}\over\sqrt{G_{33}G_{11}}})\Big[\p_3(\tilde
u_1\sqrt{G_{11}})-\p_1(\tilde u_3\sqrt{G_{33}})\Big]\\
&- {1\over G_{22}\sqrt{G_{33}}}\Big[\p_{22}\tilde u_3  - \sqrt{G_{22}}\p_{23}\tilde u_2-\p_2\tilde u_2\p_3\sqrt{G_{22}}-\p_3(\tilde u_3\p_2\sqrt{G_{22}}) \Big]\\
&- {1\over\sqrt{G_{11}G_{22}}}    \p_2( {\sqrt{G_{11}}\over\sqrt{G_{22}G_{33}}})\Big[\p_2(\tilde
u_3\sqrt{G_{33}})-\p_3(\tilde u_2\sqrt{G_{22}})\Big]\\
=&-{1\over G_{11}}\p_{11}\tilde u_3-{1\over G_{22}}\p_{22}\tilde u_3+{1\over\sqrt{G_{11}}}\p_{13}\tilde u_1+{1\over\sqrt{G_{22}}}\p_{23}\tilde u_2+\phi_3,
\endaligned
$$
where
$$\aligned
\phi_3=&{1\over G_{11}}\Big[\p_1\tilde u_1\p_3\sqrt{G_{11}}+\p_3(\tilde u_1\p_1\sqrt{G_{11}}) \Big]
+{1\over\sqrt{G}}\p_1({\sqrt{G_{22}}\over\sqrt{G_{11}}})\Big[\p_3(\tilde
u_1\sqrt{G_{11}})-\p_1\tilde u_3\Big]\\
&+ {1\over G_{22}}\Big[\p_2\tilde u_2\p_3\sqrt{G_{22}}+\p_3(\tilde u_3\p_2\sqrt{G_{22}}) \Big]
- {1\over\sqrt{G}} \p_2( {\sqrt{G_{11}}\over\sqrt{G_{22}}})\Big[\p_2\tilde
u_3-\p_3(\tilde u_2\sqrt{G_{22}})\Big].
\endaligned
$$

\subsection{Proof of Theorem \ref{Thm1}}\

\begin{proof} We only need to derive regularity near boundary. Let $x^0\in \p\O$ and we take a neighbourhood near $x^0$, and take a differentiable isomorphism to map the neighbourhood $\mU$ of $x^0$ to a domain $U$ with flat boundary $\Gamma$, and the image of $x^0$ locates in the interior of $\Gamma$.

The equation in \eqref{eqB} can be written as
\eq
\p_t\tilde u_j+\tilde a\tilde T_j+\tilde c\tilde u_j=\tilde f_j-\tilde h_j.
\eeq
Now we simplify the formula by using the condition $\div\u=0$, which gives
$$\aligned
0=&\sum_{j=1}^2\p_j\Big(\sqrt{G\over G_{jj}}\tilde u_j\Big)
+\p_3\Big(\sqrt{G\over G_{33}}\tilde u_3\Big)\\
=&\p_1(\sqrt{G_{22}}\tilde u_1)+\p_2(\sqrt{G_{11}}\tilde u_2)+\p_3(\sqrt{G}\tilde u_3)\\
=&\sqrt{G_{22}}\p_1\tilde u_1+\sqrt{G_{11}}\p_2\tilde u_2+\sqrt{G}\p_3\tilde u_3+\tilde u_1\p_1\sqrt{G_{22}}+\tilde u_2\p_2\sqrt{G_{11}}+\tilde u_3\p_3\sqrt{G}.
\endaligned
$$
Hence
\eq\label{div0}
\sqrt{G_{22}}\p_1\tilde u_1+\sqrt{G_{11}}\p_2\tilde u_2+\sqrt{G}\p_3\tilde u_3=-\tilde u_1\p_1\sqrt{G_{22}}-\tilde u_2\p_2\sqrt{G_{11}}-\tilde u_3\p_3\sqrt{G}.
\eeq
Write \eqref{div0} as
$$
{1\over
\sqrt{G}}\p_2\tilde u_2+{1\over\sqrt{G_{11}}}\p_3\tilde u_3=-{1\over G_{11}}\p_1\tilde u_1-{1\over G_{11}\sqrt{G_{22}}}[\tilde u_1\p_1\sqrt{G_{22}}+\tilde u_2\p_2\sqrt{G_{11}}+\tilde u_3\p_3\sqrt{G}].
$$
Differentiating in $y_1$ yields
\eq\label{eq1}
{1\over
\sqrt{G}}\p_{12}\tilde u_2+{1\over\sqrt{G_{11}}}\p_{13}\tilde u_3=-{1\over G_{11}}\p_{11}\tilde u_1+\zeta_1,
\eeq
where
$$\aligned
\zeta_1=&-\p_1\tilde u_1\p_1({1\over G_{11}})-\p_2\tilde u_2\p_1({1\over\sqrt{G}})-\p_3\tilde u_3\p_1({1\over\sqrt{G_{11}}})\\
&-\p_1\Big\{{1\over G_{11}\sqrt{G_{22}}}\Big[\tilde u_1\p_1\sqrt{G_{22}}+\tilde u_2\p_2\sqrt{G_{11}}+\tilde u_3\p_3\sqrt{G}\Big]\Big\}.
\endaligned
$$
Write \eqref{div0} as
$$
{1\over
\sqrt{G}}\p_1\tilde u_1+{1\over\sqrt{G_{22}}}\p_3\tilde u_3=-{1\over G_{22}}\p_2\tilde u_2-{1\over G_{22}\sqrt{G_{11}}}[\tilde u_1\p_1\sqrt{G_{22}}+\tilde u_2\p_2\sqrt{G_{11}}+\tilde u_3\p_3\sqrt{G}].
$$
Differentiating in $y_2$ yields
\eq\label{eq2}
{1\over
\sqrt{G}}\p_{12}\tilde u_1+{1\over\sqrt{G_{22}}}\p_{23}\tilde u_3=-{1\over G_{22}}\p_{22}\tilde u_2+\zeta_2,
\eeq
where
$$\aligned
\zeta_2=&-\p_1\tilde u_1\p_2({1\over\sqrt{G}})-\p_2\tilde u_2\p_{2}({1\over G_{22}})-\p_3\tilde u_3\p_2({1\over\sqrt{G_{22}}})\\
&-\p_2\Big\{{1\over G_{22}\sqrt{G_{11}}}\Big[\tilde u_1\p_1\sqrt{G_{22}}+\tilde u_2\p_2\sqrt{G_{11}}+\tilde u_3\p_3\sqrt{G}\Big]\Big\}.
\endaligned
$$
Write \eqref{div0} as
$$
{1\over \sqrt{G_{11}}}\p_1\tilde u_1+{1\over \sqrt{G_{22}}}\p_2\tilde u_2=-\p_3\tilde u_3-{1\over G}[\tilde u_1\p_1\sqrt{G_{22}}+\tilde u_2\p_2\sqrt{G_{11}}+\tilde u_3\p_3\sqrt{G}].
$$
Differentiating in $z$ yields
\eq\label{eq3}
{1\over \sqrt{G_{11}}}\p_{13}\tilde u_1+{1\over \sqrt{G_{22}}}\p_{23}\tilde u_2=-\p_{33}\tilde u_3+\zeta_3,
\eeq
where
$$
\aligned
\zeta_3=&-\p_1\tilde u_1\p_3({1\over\sqrt{G_{11}}})-\p_2\tilde u_2\p_{3}({1\over G_{22}})
-\p_3\Big\{{1\over G}\Big[\tilde u_1\p_1\sqrt{G_{22}}+\tilde u_2\p_2\sqrt{G_{11}}+\tilde u_3\p_3\sqrt{G}\Big]\Big\}.
\endaligned
$$
Plugging \eqref{eq1}, \eqref{eq2}, \eqref{eq3} into the equalities of $\tilde T_1, \tilde T_2, \tilde T_3$ respectively, we get
$$\aligned
&\tilde T_1=-{1\over G_{11}}\p_{11}\tilde u_1
-{1\over G_{22}}\p_{22}\tilde u_1-\p_{33}\tilde u_1+\zeta_1
+\phi_1,\\
&\tilde T_2=-{1\over G_{11}}\p_{11}\tilde u_2-{1\over G_{22}}\p_{22}\tilde u_2-\p_{33}\tilde u_2+\zeta_2+\phi_2,\\
&\tilde T_3=-{1\over G_{11}}\p_{11}\tilde u_3-{1\over G_{22}}\p_{22}\tilde u_3-\p_{33}\tilde u_3+\zeta_3+\phi_3.
\endaligned
$$

On $\Gamma$ we have $\tilde u_1=\tilde u_2=0$. Then from \eqref{div0} we see that
$$
\p_3\tilde u_3+H\tilde u_3=0,
$$
where $H={\p_3\sqrt{G}\over\sqrt{G}}$. So we find that $\tilde u_1,\tilde u_2$ satisfy
$$
\left\{\aligned
&\p_t\tilde u_j-{\tilde a\over G_{11}}\p_{11}\tilde u_j
-{\tilde a\over G_{22}}\p_{22}\tilde u_j-\tilde a\p_{33}\tilde u_j+\tilde c\tilde u_j=F_j\q&\text{in }(0,T)\times U,\\
&\tilde u_j=0\q&\text{on }(0,T)\times\Gamma, \\
&\tilde u_j=\tilde u_j^0\q&\text{in } U,
\endaligned\right.
$$
and $\tilde u_3$ satisfies
$$
\left\{\aligned
&\p_t\tilde u_3-{\tilde a\over G_{11}}\p_{11}\tilde u_3
-{\tilde a\over G_{22}}\p_{22}\tilde u_3-\tilde a\p_{33}\tilde u_3+\tilde c\tilde u_3=F_3\q&\text{in }(0,T)\times U,\\
&\p_3\tilde u_3+H\tilde u_3=0\q&\text{on }(0,T)\times\Gamma,\\
&\tilde u_3=\tilde u_3^0\q&\text{in } U,
\endaligned\right.
$$
where
$$F_j=\tilde f_j-\tilde h_j-\tilde a(\zeta_j+\phi_j),\qq j=1,2,3.
$$
Note that $\tilde h_j, \tilde \zeta_j, \phi_j$ contain derivatives of $\tilde u_j$ up to the first order, and contain terms involving $G_{jj}$ and their derivatives up to order $2$. Hence if $\p\O$ is of $C^{3+\a}$, then those terms are determined by $\tilde u_j$ and their first order derivatives, with coefficients that are $C^\a$ in $y_1, y_2, z$.

Note that the boundary condition for $\tilde u_3$ can be changed to a homogeneous Neumann boundary condition
if we consider a new function
$$
\hat u_3=e^{\int_0^{x_3}Hdx_3}\tilde u_3.
$$

Although the boundary conditions for $\tilde u_j$ are satisfied only on $\Gamma$, a part of the boundary of $U$, we can multiply $\tilde u_j$ by a smooth cut-off function such that $\tilde u_j$ satisfies the same type boundary condition on $\p U\setminus \Gamma$.
So we can repeat the proof of Lemmas \ref{Lem2.1} and \ref{Lem2.3}, to derive the Schauder estimates for $\tilde u_j$, $j=1,2,3$, in $(0,T]\times B^+_R$. Using the diffeomorphism $\mF$ we obtain the Schauder estimates of $\u$ in $(0,T]\times \mathcal V$, where $\mathcal V$ is a neighborhood of $\mU$ which contains the point $x^0$.

Then the conclusion of the proposition follows by covering a tubular neighbourhood of $\p\O$ with a finite number of domains as above, and using interpolation.
\end{proof}

\v0.1in

\subsection*{Acknowledgements}  This work was partially supported
by the National Natural Science Foundation of China grant no. 12071142 and 11671143, and by the research grants no. UDF01001805 and CUHKSZWDZC0003.

\v0.2in

\end{document}